\documentclass[a4paper,leqno,12pt]{amsart}
\usepackage{amsmath,amstext,amssymb,amsopn,amsthm,enumitem,float}
\usepackage[utf8]{inputenc}
\usepackage{graphicx}

\newcommand{\NN}{\mathbb{N}}
\newcommand{\CC}{\mathbb{C}}
\newcommand{\TT}{\mathbb{T}}
\newcommand{\RR}{\mathbb{R}}
\newcommand{\EE}{\mathbb{E}}
\newcommand{\PP}{\mathbb{P}}

\newcommand{\MB}{\mathcal{M}^\mathcal{B}}
\newcommand{\MD}{\mathcal{M}^\mathcal{D}}
\newcommand{\MG}{\mathcal{M}^\mathcal{G}}
\newcommand{\MDO}{\mathcal{M}^{\mathcal{D}_0}}
\newcommand{\MMR}{\mathcal{M}^\mathcal{R}}
\newcommand{\MMRO}{\mathcal{M}^{\mathcal{R}_0}}
\newcommand{\RRR}{\mathcal{R}}
\newcommand{\II}{\mathcal{I}}
\newcommand{\BB}{\mathcal{B}}
\newcommand{\DD}{\mathcal{D}}
\newcommand{\GG}{\mathcal{G}}
\newcommand{\EEE}{\mathcal{E}}
\newcommand{\JJ}{\mathcal{J}}

\newcommand{\PPP}{\mathcal{P}}

\theoremstyle{plain}
\newtheorem{theorem}{Theorem}[section]
\newtheorem{lemma}[theorem]{Lemma}
\newtheorem{fact}[theorem]{Fact}

\newtheorem*{theorem4.3}{Theorem 4.3'}
\newtheorem{proposition}[theorem]{Proposition}
\numberwithin{equation}{section}

\begin{document}

\title[Differentiation of integrals in the infinite-dimensional torus]{On differentiation of integrals in the infinite-dimensional torus}

\author{Dariusz Kosz}
\address{ \newline Dariusz Kosz
	\newline Faculty of Pure and Applied Mathematics
	\newline Wroc\l aw University of Science and Technology 
	\newline Wybrze\.ze Wyspia\'nskiego 27, 50-370 Wroc\l aw, Poland
	\newline \textit{E-mail address:} \textnormal{dariusz.kosz@pwr.edu.pl}		
}

\thanks{Research supported by the National Science Centre of Poland, project no. 2016/21/N/ST1/01496.
}

\begin{abstract} We answer the recently posed questions regarding the problem of differentiation of integrals for the Rubio de Francia basis $\RRR$ in the infinite torus $\TT^\omega$. In particular, we prove that $\RRR$ does not differentiate $L^\infty(\TT^\omega)$. Some remarks about differentiation in the context of arbitrary bases are also included.
	
	
\end{abstract}

\maketitle

\section{Introduction}

The study of the infinite torus $\TT^\omega$ arises naturally in various places as a result of efforts to extend the multidimensional analysis to the case of infinitely many dimensions. There are dozens of works dealing with this object in the context of many different branches of mathematics, including i.a. potential theory \cite{Be, B, BSC1, BSC2, BCSC}, ergodic theory \cite{L, AT} and, what is of our interest, harmonic analysis \cite{RDF2, FR1}.

Recently, Fern{\'a}ndez and Roncal \cite{FR2} have introduced a decomposition of Calder{\'o}n--Zygmund type in $\TT^\omega$ in order to prove some results on differentiation of integrals in this setting. The analysis provided there was largely inspired by the article of Rubio de Francia \cite{RDF1}, where an analogous problem in the more general context of locally compact groups was considered. We also refer the reader to previous works \cite{BF, G}, where the issue of differentiation of integrals in $\RR^n$ was widely discussed. 

The problem of differentiation of integrals always appears in connection with some differentiation basis, that is, the family of sets which are used to 'approximate' points of the space under consideration. In \cite{FR2} three types of such bases were studied: the {\it restricted Rubio de Francia basis} $\RRR_0$, the {\it Rubio de Francia basis} $\RRR$, and the {\it extended Rubio de Francia basis} $\RRR^\ast$. In the case of $\RRR_0$ some good differentiation properties are ensured by the appropriate estimate for the associated maximal operator $\MMRO$. On the other hand, the authors were able to apply the idea of Jessen \cite{J1, J2} in order to show that the corresponding result on differentiation for $\RRR^\ast$ is false. Finally, the case of $\RRR$ turned out to be more complicated and the questions regarding both differentiation and the behavior of the maximal operator $\MMR$ remained open.   

The aim of this article is to answer the questions of \cite{FR2} in the context of the Rubio de Francia basis $\RRR$. In particular, we obtain the following result, which is an immediate consequence of Theorem~\ref{thm:1} (see Section 3).

\begin{theorem}
	The Rubio de Francia basis $\RRR$ does not differentiate $L^\infty(\TT^\omega)$.
\end{theorem}

Moreover, we want to shed more light on the problem of differentiation posed for a general basis $\BB$ by providing several instructive observations and analyzing some other examples of bases. In particular, we try to show what exactly is the relationship between the problem of differentiation for $\BB$ and the properties of the operator $\MB$. We give an answer assuming a certain additional condition on $\BB$.  
 
\section{Preliminaries}

Throughout the article by $\TT$ we mean the one-dimensional torus $\{ z \in \CC : |z| = 1 \}$, which will be naturally understood as the interval $[0,1]$ with its endpoints identified (this can be done by using the relation $t \leftrightarrow e^{2 \pi i t}$, $t \in [0,1)$). 
We let $\TT^\omega$ be the product of countably many copies of $\TT$. Then $\TT^\omega$ is a compact group with identity element ${\bf 0} = (0,0,\dots)$ and the normalized Haar measure $m$, which coincides with the product of countably many copies of the Lebesgue measure $| \cdot |$ on $\TT$ (or, more precisely, on $[0,1)$). In several places we write $\TT^\omega = \TT^n \times \TT^{n, \omega}$, where $n \in \NN$ is some positive integer. Although the second object in this product is a copy of $\TT^\omega$ itself, we use the symbol $\TT^{n, \omega}$ to indicate which coordinates are considered here. Similarly, we write ${\bf 0}^{n, \omega}$ for the zero vector whenever we want it to be an element of $\TT^{n, \omega}$.  

 We say that a set $I \subset \TT^\omega$ is an {\it interval} if $I = \prod_{n \in \NN} I_n$, where for each $n \in \NN$ the set $I_n \subset \TT$ is an interval (which can be of the form $[0, \frac{1}{3}) \cup [\frac{2}{3}, 1)$, for example) and
$$
\exists_{N \in \NN} \ \forall_{n > N} \ \big( I_n = \TT \big).
$$     
The measure of $I$ is then equal to $\prod_{n \in \NN} |I_n| = \prod_{n = 1}^N |I_n|$. 

We introduce the distance between two elements of $\TT^\omega$, $g = (g_1, g_2, \dots)$ and $h = (h_1, h_2, \dots)$, by using the formula
$$
\rho(g,h) = \sum_{n=1}^\infty \frac{\min\{|g_n - h_n|, 1 - |g_n - h_n|\}}{2^n}.
$$
Then, for a measurable set $E \subset \TT^\omega$, we define its diameter 
$$ 
{\rm diam}(E) = \sup_{g,h \in E} \rho(g,h).
$$
The $\sigma$-algebra of Borel sets in $\TT^\omega$ is the smallest $\sigma$-algebra which contains the open intervals or, equivalently, the open balls with respect to $\rho$.

Suppose that for each $g \in \TT^\omega$ we have a collection $\BB(g)$ of sets of strictly positive measure whose topological
closures contain $g$. We say that a family $\{S_n : n \in \NN\}$ {\it contracts} to $g$ if $\{S_n : n \in \NN\} \subset \BB(g)$ and $\lim_{n \rightarrow \infty} {\rm diam}(S_n) = 0$; in each such case we write $S_n \Rightarrow g$. Finally, the whole family $\BB = \bigcup_{g \in \TT^\omega} \BB(g)$ equipped with the relation $\Rightarrow$ is called a {\it differentiation basis} in $\TT^\omega$ if for each $g \in \TT^\omega$ there exists a family which contracts to $g$. Throughout the article, unless otherwise stated, we deal with the bases of non-centered type, that is, we assume (without further mention) that $\BB(g) = \{ B \in \BB : g \in \overline{B} \}$ holds for each $g \in \TT^\omega$.

For an integrable function $f \in L^1(\TT^\omega)$ and a set $E \subset \TT^\omega$ satisfying $m(E) > 0$ we denote the average value of $f$ on $E$ by $f_E$, that is,
$$
f_E = \frac{1}{m(E)} \int_E f \, dm.
$$
Then, given a basis $\BB$ let us define the associated maximal operator $\MB$ by
$$
\MB f(g) = \sup_{B \in \BB(g)} |f_B|, \qquad g \in \TT^\omega,
$$        
for each $f \in L^1(\TT^\omega)$. Finally, we also introduce the following truncated operator
$$
\MB_{r_0} f(g) = \sup_{\substack{B \in \BB(g) \\ {\rm diam}(B) < r_0}} |f_B|, \qquad g \in \TT^\omega,
$$
where $r_0 > 0$ is some fixed positive number.

It is usually an important issue to study mapping properties of maximal operators. In the present work we are particularly interested in the weak type $(1,1)$ inequality. To be precise, the operator $\MB$ is said to be of {\it weak type} $(1,1)$ if
$$
\lambda m(\{ g \in \TT^\omega : \MB f(g) > \lambda \}) \leq C \|f\|_1, \qquad f \in L^1(\TT^\omega), \ \lambda > 0,
$$ 
holds for some numerical constant $C > 0$ independent of $f$ and $\lambda$.

It is well known that in many situations there are deep connections between the existence of a limit of certain sequence of operators and the behavior of the associated maximal operator (see \cite{St}, for example). This is the case also for the infinite torus and the averaging operators $\mathcal{A}_E f = f_E$, $E \subset \TT^\omega$. Namely, if $\MB$ is of weak type $(1,1)$, then one can obtain the following analogue of the Lebesgue differentiation theorem. 

\begin{fact}\label{fact:1}
{\rm (cf. \cite[Theorem 1.1]{GW})} Let $\BB$ be an arbitrary differentiation basis. If $\MB$ is of weak type $(1,1)$, then for each $f \in L^1(\TT^\omega)$ and a.e. $g \in \TT^\omega$ (the set of $g$'s may depend on $f$) we have
\begin{equation}\tag{D}\label{cond:1}
\Big( \{S_n : n \in \NN \} \subset \BB \ \wedge \ S_n \Rightarrow g \Big) \implies \Big( \lim_{n \rightarrow \NN} f_{S_n} = f(g) \Big).
\end{equation}	
\end{fact}	
 
\noindent Given a basis $\BB$ we say that $\BB$ {\it differentiates} $L^1(\TT^\omega)$ if \eqref{cond:1} holds for each $f \in L^1(\TT^\omega)$ and a.e. $g \in \TT^\omega$. Similarly, $\BB$ differentiates $L^p(\TT^\omega)$, $p \in (1, \infty]$, if \eqref{cond:1} holds for each $f \in L^p(\TT^\omega)$ and a.e. $g \in \TT^\omega$. Observe that for $1 \leq p_1 < p_2 \leq \infty$ we have the inclusion $L^{p_2}(\TT^\omega) \subset L^{p_1}(\TT^\omega)$. Thus, if $\BB$ differentiates $L^{p_1}(\TT^\omega)$, then $\BB$ differentiates $L^{p_2}(\TT^\omega)$. 

Let us now briefly describe the Rubio de Francia bases, restricted $\RRR_0$ and non-restricted $\RRR$, which were discussed in \cite{FR2} in the context of differentiation. For each $k \in \NN$ we set $R_k = \{0, \frac{1}{k}, \dots, \frac{k-1}{k} \}$. Then, given $n \in \NN$ we define the finite subgroup $H_n \subset \TT^\omega$ and the open set $V_n \subset \TT^\omega$ by using the following scheme (see \cite{FR2} for more details).

\begin{table}[H]
\caption{The scheme of the objects used to define the Rubio de Francia bases}
	\begin{center}
		\begin{tabular}{ r l l }
			\hline
			$n$ & $H_n$ & $V_n$ \\
			\hline
			$1$ & $R_2 \times \{{\bf 0}^{1, \omega}\}$ 
			& $(0,\frac{1}{2}) \times \TT^{1, \omega}$ \\ 
			$2$ & $R_2 \times R_2 \times \{{\bf 0}^{2, \omega}\}$ 
			& $(0,\frac{1}{2}) \times (0,\frac{1}{2}) \times \TT^{2, \omega}$ \\ 
			$3$ & $R_4 \times R_2 \times \{{\bf 0}^{2, \omega}\}$ & $(0,\frac{1}{4}) \times (0,\frac{1}{2}) \times \TT^{2, \omega}$ \\
			$4$ & $R_4 \times R_4 \times \{{\bf 0}^{2, \omega}\}$ & $(0,\frac{1}{4}) \times (0,\frac{1}{4}) \times \TT^{2, \omega}$ \\ 
			$5$ & $R_4 \times R_4 \times R_2 \times \{{\bf 0}^{3, \omega}\}$ & $(0,\frac{1}{4}) \times (0,\frac{1}{4}) \times (0,\frac{1}{2}) \times\TT^{3, \omega}$ \\ 
			$6$ & $R_4 \times R_4 \times R_4 \times \{{\bf 0}^{3, \omega}\}$ & $(0,\frac{1}{4}) \times (0,\frac{1}{4}) \times (0,\frac{1}{4}) \times\TT^{3, \omega}$ \\
			$7$ & $R_8 \times R_4 \times R_4 \times \{{\bf 0}^{3, \omega}\}$ & $(0,\frac{1}{8}) \times (0,\frac{1}{4}) \times (0,\frac{1}{4}) \times\TT^{3, \omega}$ \\ 
			$8$ & $R_8 \times R_8 \times R_4 \times \{{\bf 0}^{3, \omega}\}$ & $(0,\frac{1}{8}) \times (0,\frac{1}{8}) \times (0,\frac{1}{4}) \times\TT^{3, \omega}$ \\ 
			$9$ & $R_8 \times R_8 \times R_8 \times \{{\bf 0}^{3, \omega}\}$ & $(0,\frac{1}{8}) \times (0,\frac{1}{8}) \times (0,\frac{1}{8}) \times\TT^{3, \omega}$ \\
			$10$ & $R_8 \times R_8 \times R_8 \times R_2 \times \{{\bf 0}^{4, \omega}\}$ & $(0,\frac{1}{8}) \times (0,\frac{1}{8}) \times (0,\frac{1}{8}) \times (0,\frac{1}{2}) \times\TT^{4, \omega}$ \\
			$\cdots$ & $\cdots$ & $\cdots$ \\
			\hline
		\end{tabular}
	\end{center}
\end{table}
\noindent Finally, we let
$$
\RRR_0 = \{ g + V_n : n \in \NN, \, g \in H_n \} \quad (\textit{restricted Rubio de Francia basis}),
$$
and
$$
\RRR = \{ g + V_n : n \in \NN, \, g \in \TT^\omega \} \quad (\textit{Rubio de Francia basis}).
$$
Notice that both bases are considered as bases of non-centered type.   

It was shown \cite[Theorem~9]{FR2} that in the case of $\RRR_0$ the associated maximal operator is of weak type $(1,1)$ (cf. \cite[Theorem~2.10]{D}). Consequently, $\RRR_0$ differentiates $L^1(\TT^\omega)$. On the other hand, little is known so far about the case of $\RRR$. The following questions posed in \cite{FR2} have remained open: 
  
\begin{enumerate}[label=(Q\arabic*)]
\item Can one deduce that $\MMR$ is of weak type $(1,1)$ assuming that $\RRR$ differentiates $L^1(\TT^\omega)$? 
\item Is $\MMR$ of weak type $(1,1)$?
\item Does $\RRR$ differentiate $L^\infty(\TT^\omega)$?	
\end{enumerate}

The rest of the paper is devoted to discuss the three questions mentioned above. Namely, in Section 3 we prove that the answers to both (Q2) and (Q3) are negative (thus, in particular, we see that there is no longer any reason to consider (Q1) in its present form). On the other hand, in Section 4 we deal with some issue related to (Q1), but introduced in the context of arbitrary bases $\BB$. 

\section{On the Rubio de Francia basis $\RRR$}

In this section we deal with the Rubio de Francia basis $\RRR$ introduced before. As mentioned earlier, our aim is to show that the answers to (Q2) and (Q3) are negative. It will also be convenient to formulate the following additional problem:

\begin{enumerate}[label=(Q3')]
	\item Does $\RRR$ differentiate $L^1(\TT^\omega)$?	
\end{enumerate}

Let us point out that the following implications between the answers to the three questions that we are interested in hold:
$$
\Big( {\rm ANS(Q3) = NO} \Big) \implies \Big( {\rm ANS(Q3') = NO} \Big)  \implies \Big( {\rm ANS(Q2) = NO} \Big).
$$
At this point one might observe that it is enough to show that $\RRR$ does not differentiate $L^\infty(\TT^\omega)$. However, it seems to be more instructive to examine all the three problems directly, starting with the easiest one. Thus, the first result that we show here is the following.

\begin{proposition}\label{obs:1}
	$\MMR$ is not of weak type $(1,1)$
\end{proposition}

\noindent Indeed, given $n \in \NN$ we take $\epsilon_n \in (0, \frac{1}{2^{n+1}})$ and consider $f_n = \frac{1}{m(A_n)} \chi_{A_n}$, where
$$
A_n = \Big(\frac{1}{2} - \epsilon_n, \frac{1}{2} + \epsilon_n \Big)^n \times \TT^{n,\omega}.
$$
We denote
$$
U_n = V_{n^2} = \Big(0, \frac{1}{2^n} \Big)^n \times \TT^{n, \omega} 
$$
and
$$
E_n = \Big(\frac{1}{2} - \frac{1}{2^n} + \epsilon_n, \frac{1}{2} + \frac{1}{2^n} - \epsilon_n \Big)^n \times \TT^{n, \omega}.
$$
Observe that for each $x \in (\frac{1}{2} - \frac{1}{2^n} + \epsilon_n, \frac{1}{2} + \frac{1}{2^n} - \epsilon_n \Big)$ it is possible to find an interval $(a, b) \subset (0, 1)$ such that $b - a = \frac{1}{2^n}$ and $\{x\} \cup (\frac{1}{2} - \epsilon_n, \frac{1}{2} + \epsilon_n ) \subset (a, b)$. Hence, applying this argument $n$ times, we find that for each $g \in E_n$ there exists $h \in \TT^\omega$ satisfying
$$
\{g\} \cup A_n \subset h + U_n.
$$
Therefore, we conclude that
$$
\MMR f_n(g) \geq \frac{\|f_n\|_1}{m(h + U_n)} = 2^{n^2}
$$
holds for each $g \in E_n$, and consequently
\begin{equation}\label{eq:1}
\frac{2^{n^2-1} m(\{g : |\MMR f_n(g) | > 2^{n^2-1} \} ) }{\|f_n\|_1} \geq 2^{n^2-1} m(E_n) 
= 2^{n-1} (1 - 2^n \epsilon_n)^n.
\end{equation}
If $\epsilon_n$ is sufficiently small, then the right hand side of \eqref{eq:1}  is bounded from below by $2^{n-2}$. Thus, since $n \in \NN$ is arbitrary, we see that $\MMR$ is not of weak type $(1,1)$.
\medskip

Our second goal is to answer the additional question (Q3').

\begin{proposition}\label{obs:2}
$\RRR$ does not differentiate $L^1(\TT^\omega)$	
\end{proposition}

\noindent Indeed, given $n \in \NN$ we take $\epsilon_n \in (0, \frac{1}{2^{n+1}})$ and denote
$$
\PPP_n = \Big\{ \frac{1}{2^n}, \frac{3}{2^n}, \frac{5}{2^n}, \dots, \frac{2^n-1}{2^n}  \Big\}^n \subset (0,1)^n.
$$
For each $P = (P_1, \dots, P_n) \in \PPP_n$ we consider the sets
$$
A_P = (P_1 - \epsilon_n, P_1 + \epsilon_n) \times \dots \times (P_n - \epsilon_n, P_n + \epsilon_n) \times \TT^{n, \omega}
$$
and
$$
E_P = \Big(P_1 - \frac{1}{2^n} + \epsilon_n, P_1 + \frac{1}{2^n} - \epsilon_n\Big) \times \dots \times \Big(P_n - \frac{1}{2^n} + \epsilon_n, P_n + \frac{1}{2^n} - \epsilon_n\Big) \times \TT^{n, \omega}.
$$
Let us define $f_n \in L^1(\TT^\omega)$ by
$$
f_n = \sum_{P \in \PPP_n} \frac{2^{-n^2}}{(2\epsilon_n)^n} \chi_{A_P}.
$$
Since $\PPP_n$ consists of precisely $2^{n^2 - n}$ elements, we see that $\|f_n\|_1 = 2^{-n}$.

Fix $P \in \PPP_n$ and let $g \in E_P$. We can find $h \in \TT^\omega$ such that
$$
\{ g \} \cup A_P \subset h + U_n
$$ 
(here $U_n = V_{n^2}$ is defined as in Proposition~\ref{obs:1}) and hence
$$
(f_n)_{h + U_n} \geq \frac{2^{-n^2}}{m(U_n)} = 1.
$$
Denote $A_n = \bigcup_{P \in \PPP_n} A_P$ and $E_n = \bigcup_{P \in \PPP_n} E_P$. Observe that
\begin{equation}\label{eq:2}
m(A_n) = 2^{n^2 - n} (2 \epsilon_n)^n \leq 2^{-n}.
\end{equation}
Moreover, if $\epsilon_n$ is chosen to be sufficiently small, then
\begin{equation}\label{eq:3}
m(E_n) = 2^{n^2 - n} \Big(2 (2^{-n} - \epsilon_n)\Big)^n = (1 - 2^n \epsilon_n)^n > 1 - 2^{-n}.
\end{equation}

We now let
$$
f = \sum_{n \in \NN} f_n,
$$
where for each $n \in \NN$ the corresponding parameter $\epsilon_n$ is such that \eqref{eq:2} and \eqref{eq:3} hold. Note that $f \in L^1(\TT^\omega)$, since $\|f\|_1 = \sum_{n \in \NN} \|f_n\|_1 = 1$. Let
$$
A = \bigcup_{n \in \NN} A_n, \qquad E = \bigcap_{n \in \NN} \bigcup_{k \geq n} E_k.
$$
One can easily show that $m(A) \leq \frac{1}{2}$. Moreover, since $\bigcap_{k = n}^\infty E_k \subset E$ and $m(\bigcap_{k = n}^\infty E_k) \geq 1-2^{n-1}$ hold for each $n \in \NN$, we conclude that $m(E) = 1$. Note that for each $g \in \TT^\omega \setminus A$ we have $f(g) = 0$. On the other hand, for each $g \in E$ there exists a sequence $(h_n )_{n \in \NN} \subset \TT^\omega$ such that $g \in h_n + U_n$, $n \in \NN$, and
$$
\limsup_{n \in \NN} f_{h_n + U_n} \geq \limsup_{n \in \NN} (f_n)_{h_n + U_n} \geq 1.
$$
Consequently, for each $g \in E \setminus A$,
$$
\limsup_{n \in \NN} f_{h_n + U_n} > f(g). 
$$
Thus, we conclude that the Rubio de Francia basis $\RRR$ does not differentiate $L^1(\TT^\omega)$.


\medskip

Now, it remains to show that $\RRR$ does not differentiate $L^\infty(\TT)$. Let us remark that the crucial fact used to justify Propositions~\ref{obs:1}~and~\ref{obs:2} was that the basis $\RRR$ consists of intervals which can additionally be translated by an arbitrary element of the group. The exact shape of these intervals is of less importance here. It turns out that the same is true for the last result we are interested in. Namely, we will show that the answer to (Q3) is negative with $\RRR$ replaced by any collection $\BB$ satisfying the following assertions:
\begin{enumerate}[label=(B\arabic*)]
	\item each element of $\BB$ is an interval, 
	\item $\BB$ is translation-invariant, that is, if $B \in \BB$, then $\{ g + B : g \in \TT^\omega \} \subset \BB$,
	\item for any $\epsilon > 0$ there exists $B \in \BB$ such that ${\rm diam}(B) < \epsilon$. 
\end{enumerate}
Note that if (B1)--(B3) holds, then $\BB$ is indeed a differentiation basis.

Let us now formulate the main result of this section. 

\begin{theorem}\label{thm:1}
	Fix $\epsilon > 0$ and let $\BB$ be an arbitrary collection of sets in $\TT^\omega$ satisfying {\rm (B1)--(B3)}. Then there exist sets $A, E \subset \TT^\omega$ such that
	\begin{equation}\label{eq:10}
	m(A) < \epsilon \quad \textit{and} \quad  m(E) = 1,
	\end{equation}
	and for each $g \in E$ there exists a family $\{Q_n : n \in \NN \} \subset \BB$ such that $Q_n \Rightarrow g$ and 
	\begin{equation}\label{eq:11}
	\limsup_{n \rightarrow \infty} \, (\chi_A)_{Q_n} \geq e^{-8} > 0.
	\end{equation}
	In particular, $\BB$ does not differentiate $L^\infty(\TT^\omega)$.
\end{theorem}

\noindent The proof of Theorem~\ref{thm:1} will be preceded by several auxiliary lemmas. 
 
\begin{lemma}\label{lem:1}
Fix $n \in \NN$ and let $(\alpha_i)_{i=1}^{n^2}$ be a sequence of strictly positive numbers. Denote
$$
\II_n = (0, \alpha_1) \times \dots \times  (0, \alpha_{n^2}) \subset [0, \infty)^{n^2}
$$
and
$$
\JJ_n = \Big\{ (x_1, \dots, x_{n^2}) \in \Big(1+ \frac{1}{n}\Big) \II_n : \# \Big( \Big\{ x_i : x_i \in \Big[\alpha_i, \Big(1+\frac{1}{n}\Big) \alpha_i \Big)     \Big\} \Big) \geq 4 n \Big\},
$$
where $\Big(1+ \frac{1}{n}\Big) \II_n = \prod_{i=1}^{n^2} (0, (1+\frac{1}{n})\alpha_i)$ is the dilation of $\II_n$ with respect to the origin and $\#( \, \cdot \, )$ is the counting measure. Then
\begin{equation}\label{eq:4}
\frac{|\JJ_n |}{|(1+ \frac{1}{n}) \II_n| } < \frac{1}{2}.
\end{equation}	
\end{lemma}

\begin{proof}
	Observe that the quantity on the left hand side of \eqref{eq:4} is equal to the probability
$$
\PP(X \geq 4n),
$$
where $X$ 
is a binomially distributed random variable with parameters $n^2$ and $\frac{1}{n+1}$. We note that $\EE(X) = \frac{n^2}{n+1}$ and $\rm{Var}(X) = \frac{n^3}{(n+1)^2}$. By applying Chebyshev's inequality we get
$$
\PP(X \geq 4n) \leq \PP\Big( \Big|X - \frac{n^2}{n+1} \Big| \geq  \frac{3 n^3}{(n+1)^2} \Big) \leq \frac{(n+1)^2}{9 n^3} < \frac{1}{2}
$$
and hence \eqref{eq:4} is satisfied.
\end{proof}

\begin{lemma}\label{lem:2}
	For fixed $n \in \NN \setminus \{1\}$ let $(\alpha_i)_{i=1}^{n^2}$, $\II_n$, and $\JJ_n$ be as in Lemma~\ref{lem:1}. 
Then for each $x = (x_1, \dots, x_{n^2}) \in (1+\frac{1}{n})\II_n \setminus \JJ_n$ there exists $y = (y_1, \dots, y_{n^2}) \in \RR^{n^2}$ such that
\begin{equation}\label{eq:5}
x \in y + \II_n \quad \textit{and} \quad |(\II_n + y) \cap \II_n | > e^{-8} |\II_n|. 
\end{equation}
\end{lemma}
\begin{proof}
	Let $x \in (1+\frac{1}{n})\II_n \setminus \JJ_n$. We define $y \in \RR^{n^2}$ by letting
$$
y_i = \left\{ \begin{array}{rl}
0 & \textrm{if }  x_i \in (0, \alpha_i),   \\
\frac{\alpha_i}{n} & \textrm{otherwise,} \end{array} \right.
$$
for each $i \in \{ 1, \dots, n^2\}$. Obviously, $x \in y + \II_n$. Moreover, the ratio $$
\frac{|(0, \alpha_i) \cap (y_i, \alpha_i+y_i) |}{|(0, \alpha_i)|}
$$
equals $1$ if $y_i = 0$ or $\frac{n-1}{n}$ if $y_i = \frac{\alpha_i}{n}$. Consequently, we have
$$
\frac{|(\II_n + y) \cap \II_n |}{|\II_n|} = \Big( \frac{n-1}{n} \Big)^{\#\{ x_i : x_i \in [\alpha_i, (1+1/n) \alpha_i) \}} > \Big( \frac{n-1}{n} \Big)^{4n} \geq \Big( \frac{n-1}{n} \Big)^{8(n-1)} > e^{-8}
$$
and hence \eqref{eq:5} is satisfied. 
\end{proof}

\begin{lemma}\label{lem:3}
Fix $K, L \in \NN$ and let $\BB$ be an arbitrary collection of sets in $\TT^\omega$ satisfying {\rm (B1)--(B3)}. Then there exist sets $A_{K,L}$ and $E_{K,L}$ of the form
\begin{equation}\label{eq:6}
A_{K,L} = \TT^K \times A_{K,L}^\circ \times \TT^{K + (L+1)^2, \omega} \subset \TT^\omega
\end{equation} 
and 
\begin{equation}\label{eq:7}
E_{K,L} = \TT^K \times E_{K,L}^\circ \times \TT^{K + (L+1)^2, \omega}  \subset \TT^\omega,
\end{equation}
where $  A_{K,L}^\circ, E_{K,L}^\circ \subset [0,1)^{(L+1)^2}$, such that
\begin{equation}\label{eq:8}
m(A_{K,L}) < e^{-L} \quad \textit{and} \quad m(E_{K,L}) > \frac{1}{2e},
\end{equation}
and
\begin{equation}\label{eq:9}
\MB_{2^{-L}} \chi_{A_{K,L}}(g) > e^{-8}, \qquad g \in E_{K,L}.
\end{equation}
\end{lemma}

\begin{proof}
By (B1)--(B3) we can find an interval $Q \in \BB$ of the form
$$
Q = I_1 \times \dots \times I_k \times \TT^{k, \omega} 
$$	
for some $k \in \NN$, satisfying
\begin{equation}\label{diamQ}
{\rm diam}(Q) \leq \frac{1}{2^{K+(L+1)^2}((L+1)^2+1)} \cdot \frac{L+1}{L+2}
\end{equation}
and such that $I_i \in \{(0, r_i), (0, r_i], [0, r_i), [0, r_i]\}$,  $r_i \in (0,1]$, for each $i \in \{1, \dots, k\}$. Note that $k \geq K+(L+1)^2$, since ${\rm diam}(Q) \geq \frac{1}{2} \cdot 2^{-(k+1)}$. Moreover, given $i \leq K+(L+1)^2$, we deduce from \eqref{diamQ} that $r_i$ satisfies 
\begin{equation}\label{condQ}
(1 + \frac{1}{L+1}) r_i \leq \frac{1}{(L+1)^2 + 1}.
\end{equation}

Now for each $i \in \{K+1, \dots, K + (L+1)^2 \}$ consider the set
$$
\PPP_i = \Big\{ j \cdot \Big(1+\frac{1}{L+1} \Big)r_i : j = 0, \dots, l_i -1 \Big\} \subset [0, 1),
$$
where $l_i \in \NN$ is such that
$$
l_i \cdot \Big(1+\frac{1}{L+1}\Big)r_i \leq 1 < (l_i+1) \cdot \Big(1+\frac{1}{L+1}\Big)r_i.
$$
Next, denote
$$
\PPP^\circ = \PPP_{K+1} \times \dots \times \PPP_{K + (L+1)^2} \subset [0,1)^{(L+1)^2}. 
$$
We define $A_{K,L}$ by taking
$$
A_{K,L}^\circ = \bigcup_{p \in \PPP^\circ} p + \II^\circ
$$ 	
in \eqref{eq:6}, where $\II^\circ$ is the set $\II_n$ introduced in Lemma~\ref{lem:1} for $n = L+1$ and $\alpha_i = r_{K+i}$, $i \in \{1, \dots, (L+1)^2 \}$. Similarly, we define $E_{K,L}$ by taking
$$
E_{K,L}^\circ = \bigcup_{p \in \PPP^\circ} p + \Big(\Big(1+\frac{1}{n}\Big)\II^\circ \setminus \JJ^\circ \Big)
$$ 	
in \eqref{eq:7}, where $\JJ^\circ$ is the set $\JJ_n$ introduced in Lemma~\ref{lem:1} for the same parameters as before. We shall prove that \eqref{eq:8} and \eqref{eq:9} hold for this choice of $A_{K,L}$ and $E_{K,L}$.

First, let us observe that \eqref{condQ} implies
$$
m(A_{K,L}) \leq \Big(\frac{L+1}{L+2} \Big)^{(L+1)^2} < \Big(\frac{L+1}{L+2} \Big)^{(L+2)L} < e^{-L}.
$$
Moreover, since the sets $ p + \Big(1+\frac{1}{L+1}\Big)\II^\circ  $, $p \in \PPP^\circ$, are disjoint, by Lemma~\ref{lem:1}~and~\eqref{condQ}
\begin{align*}
m(E_{K,L}) &= \Big|  \bigcup_{p \in \PPP^\circ}  p + \Big(\Big(1+\frac{1}{L+1}\Big)\II^\circ \setminus \JJ^\circ \Big)  \Big| > 
\frac{1}{2} \Big|  \bigcup_{p \in \PPP^\circ}  p + \Big(1+\frac{1}{L+1}\Big)\II^\circ   \Big| \\
&\geq \frac{1}{2} \prod_{i = K+1}^{K + (L+1)^2 } \Big( 1 - \Big(1 + \frac{1}{L+1}\Big) r_{i} \Big) \geq \frac{1}{2} \Big( 1 - \frac{1}{(L+1)^2 + 1} \Big)^{(L+1)^2} > \frac{1}{2e}
\end{align*}
and thus \eqref{eq:8} is satisfied.

Let us now fix $g = (g_1, g_2, \dots ) \in E_{K,L}$. Then
$$
x = \big(g_{K+1}, \dots, g_{K+(L+1)^2}\big) \in p + \Big(\Big(1+\frac{1}{n}\Big)\II^\circ \setminus \JJ^\circ \Big)
$$
for some $p \in \PPP^\circ$. By Lemma~\ref{lem:2} there exists $y \in \RR^{(L+1)^2}$ such that $x \in y + p + \II^\circ$ and
$$
|(y+p+\II^\circ) \cap A_{K,L}^\circ | \geq |(y+p+\II^\circ) \cap (p + \II^\circ) | > e^{-8} |\II^\circ|.
$$ 
Consequently, there exists $h \in \TT^\omega$ satisfying $g \in h + Q$ and $(\chi_{A_{K,L}})_{h + Q} > e^{-8}$. Finally, by \eqref{diamQ} we have ${\rm diam} (h + Q) < 2^{-L}$, which justifies \eqref{eq:9}.
\end{proof}

We are now ready to prove Theorem~\ref{thm:1}.

\begin{proof}[Proof of Theorem~\ref{thm:1}]
Let us put
$$
A = \bigcup_{n \in \NN} A_n \quad {\rm and} \quad E = \bigcap_{n\in \NN} \bigcup_{k \geq n} E_k,  
$$
where the pairs $\{A_n, E_n\}$, $n \in \NN$, are constructed inductively in the following way. First, take $K_1 = 1$ and $L_1 \in \NN$ such that $e^{-L_1} \leq \frac{\epsilon}{2}$. We let $A_1$ and $E_1$ be the sets $A_{K,L}$ and $E_{K,L}$ from Lemma \ref{lem:3}, respectively, for $K = K_1$ and $L = L_1$. In the second step, let us assume that given $n \in \NN$ we have already chosen $K_i$, $L_i$, $A_i$, and $E_i$ for each $i \in \{1, \dots, n\}$. Then we take $K_{n+1}, L_{n+1} \in \NN$ satisfying $K_{n+1} > K_n + (L_n + 1)^2$ and $e^{-L_{n+1}} \leq \frac{\epsilon}{2^{n+1}}$. Finally, we let $A_{n+1}$ and $E_{n+1}$ to be the sets $A_{K,L}$ and $E_{K,L}$ from Lemma \ref{lem:3}, respectively, for $K = K_{n+1}$ and $L = L_{n+1}$. We shall prove that \eqref{eq:10} and \eqref{eq:11} hold for this choice of $A$ and $E$. 

First, it is easy to see that by \eqref{eq:8} we have
$$
m(A) \leq \sum_{n\in \NN} m(A_n) < \sum_{n \in \NN} e^{-L_n} \leq \sum_{n \in \NN} \frac{\epsilon}{2^n} = \epsilon.  
$$
Next, notice that $m(E_n) > \frac{1}{2e}$ for each $n \in \NN$. Moreover, observe that the sets $E_n$, $n \in \NN$, are independent in the sense that for each $k \in \NN$ and pairwise different indices $n_1, \dots, n_k \in \NN$ we have
$$
m\Big(\bigcap_{i=1}^k E_{n_i}\Big) = \prod_{i=1}^k m(E_{n_i}). 
$$
Indeed, the above equality follows from \eqref{eq:7} and the fact that $K_{n+1} > K_n + (L_n + 1)^2$ holds for each $n \in \NN$. By applying the second Borel--Cantelli lemma we conclude that $m(E) = 1$ and therefore \eqref{eq:10} is satisfied.

Now, let us take $g \in E$. There exists a strictly increasing sequence $(k_n)_{n \in \NN}$ satisfying $g \in E_{k_n}$ for each $n \in \NN$. In view of \eqref{eq:9} we conclude that for each $n \in \NN$ there exists $g \in Q_n \in \BB$ such that $(\chi_A)_{Q_n} > e^{-8}$ and ${\rm diam}(Q_n) < 2^{-L_{k_n}}$. In particular, we obtain that \eqref{eq:11} holds. Finally, since $\lim_{n \rightarrow \infty} L_{k_n} = \infty$, we see that $Q_n \Rightarrow g$. Consequently, $\{Q_n : n \in \NN \}$ is as desired. 
\end{proof} 

\section{Differentiation on $\TT^\omega$}

In Section 3 we provided answers to (Q2) and (Q3), the questions formulated at the and of Section 2. Now we will take a closer look at the issue that naturally arises from (Q1). Namely, given a differentiation basis $\BB$ we would like to evaluate whether $\BB$ differentiates $L^1(\TT^\omega)$ by looking at the properties of $\MB$. 

Let us recall Fact \ref{fact:1} which says that if $\MB$ is of weak type $(1,1)$, then $\BB$ differentiates $L^1(\TT^\omega)$. Our first result shows that the opposite implication cannot be expected to be true in general.

\begin{proposition}\label{obs:3}
There exists a basis $\DD$ such that the following assertions are satisfied:
\begin{enumerate}[label=\rm(\roman*)]
	\item $\DD$ differentiates $L^1(\TT^\omega)$,
	\item $\MD$ is not of weak type $(1,1)$. 
\end{enumerate}
\end{proposition}

\noindent Indeed, for each $n \in \NN$ let $U_n = V_{n^2} = (0,2^{-n})^n \times \TT^{n,\omega}$. We denote
$$
U_{n,i} = U_n \cup (e_{n,i} + U_n), \qquad n \in \NN, \ i \in \{1, \dots, n\},
$$
where $e_{n,i} \in \TT^\omega$ has all coordinates zero except the $i$th which is $2^{-n}$ (here one could also replace $U_{n,i}$ with ${\rm int} (\overline{U_{n,i}})$, which is open and connected). We let
$$
\DD = \RRR_0 \cup \DD_0,
$$
where $\RRR_0$ is the restricted Rubio de Francia basis, while
$$
\DD_0 = \{ U_{n,i} :  n \in \NN, \ i \in \{1, \dots, n\} \}.
$$
We shall show that (i) and (ii) are satisfied for this choice of $\DD$. 

Let $f \in L^1(\TT^\omega)$. Recall that the basis $\RRR_0$ differentiates  $L^1(\TT^\omega)$. Thus, there exists a set $L^f \subset \TT^\omega$ such that $m(L_f) = 1$ and 
$$
\forall_{g \in L_f} \forall_{R_n \Rightarrow g} \big(\lim_{n \rightarrow \infty} f_{R_n} = f(g) \big).
$$
Let us now fix $g \in L_f \setminus \{ \bf 0\}$ and consider a family $\{D_n : n \in \NN\}$ which contracts to $g$. Observe that the closure of each element of $\DD_0$ contains $\bf 0$. This fact implies that there exists $N \in \NN$ such that for each $n \geq N$ we have $D_n \notin \DD_0$ and, consequently, $D_n \in \RRR_0$. Therefore, $\lim_{n \in \NN} f_{D_n} = f(g)$ holds. Finally, since $f \in L^1(\TT^\omega)$ was arbitrary and $m(L_f \setminus \{ {\bf 0} \}) = 1$, we conclude that $\DD$ differentiates $L^1(\TT^\omega)$. 

Now we show that $\MDO$ (and hence $\MD$) is not of weak type $(1,1)$. Given $n \in \NN$ we take $f_n = \chi_{U_n} \in L^1(\TT^\omega)$. Observe that for each $g \in U_{n,i} \setminus U_n$, $i \in \{1, \dots, n\}$, we have
$$
\MDO f_n(g) \geq (f_n)_{U_{n,i}} = \frac{m(U_n)}{m(U_{n,i})} = \frac{1}{2}.
$$ 
Consequently,
$$
\frac{1}{3} \, m \Big( \Big\{ g \in \TT^\omega : \MDO f_n(g) > \frac{1}{3} \Big\} \Big) \geq \frac{1}{3} m\Big(\bigcup_{i=1}^n U_{n,i}\Big) = \frac{n+1}{3} m(U_n) = \frac{n+1}{3} \|f_n\|_1. 
$$
Therefore, since $n \in \NN$ was arbitrary, we conclude that $\MDO$ is not of weak type $(1,1)$.
\medskip

The idea behind the construction of $\DD$ is that at any point $g \neq {\bf 0}$ the operator $\MD$ behaves locally like an operator with good mapping properties. The behavior near ${\bf 0}$ makes $\MD$ not of weak type $(1,1)$, while the problem of differentiating remains unaffected. Motivated by this remark, in the context of an arbitrary basis $\BB$ we introduce quantities that allow us to measure the sizes of the sets on which $\MB$ behaves badly.  

For a given basis $\BB$ and each $k \in \NN$ we denote $\delta_k^\BB = \sup_{E \in \EEE_k} m(E)$, where
\begin{equation*}\label{delta}
\EEE_k = \Big\{ E \subset \TT^\omega : \exists_{f \in L^1(\TT^\omega)} \  \exists_{\lambda > 0} \ \big( \forall_{x \in E} \ \MB f(x) > \lambda \big) \wedge \big( \lambda  m(E) > 2^k \|f\|_1 \big) \Big\}.
\end{equation*}
Observe that if $\MB$ is of weak type $(1,1)$, then there exists $k_0$ such that $\delta_k^\BB = 0$ for each $k \geq k_0$. In fact, it will be proved later on that the condition $\lim_{k \rightarrow \infty} \delta_k^\BB = 0$ is enough to ensure that $\BB$ differentiates $L^1(\TT^\omega)$. At the first glance, one would even expect that this condition is sufficient and necessary at the same time. Unfortunately, this is not true, as the example below shows. 

\begin{proposition}\label{obs:4}
There exists a basis $\GG$ such that the following assertions are satisfied:
\begin{enumerate}[label=\rm(\roman*')]
	\item $\GG$ differentiates $L^1(\TT^\omega)$,
	\item $\delta_k^\GG = 1$ for each $k \in \NN$. 
\end{enumerate}	
\end{proposition}

\noindent Indeed, let
$$
\GG = \RRR_0 \cup \GG_0,
$$
where $\RRR_0$ is the restricted Rubio de Francia basis, while
$$
\GG_0 = \{ V_n \cup (h_n + V_n) : n \in \NN, \, h_n \in H_n \}. 
$$
We shall show that (i') and (ii') are satisfied for this choice of $\GG$.

First, notice that (i') can be verified by invoking the argument which was used in Proposition~\ref{obs:3} to obtain (i) for the basis $\DD$. Indeed, it suffices to observe that the closure of each element of $\GG_0$ contains ${\bf 0}$. 

Let us now fix $k \in \NN$ and take $f_k = \chi_{V_{k+2}}$. Note that $m(V_{k+2}) = 2^{-(k+2)}$ and hence $\frac{1}{3} > 2^k \|f_k\|_1$. Since $\MG f_k(g) > \frac{1}{3}$ holds for almost every $g \in \TT^\omega$, we conclude that $\delta_k^\GG = 1$.      
\medskip

It is instructive to look closer at the structure of $\GG$, in order to indicate where its degeneracy lies. Observe that, in particular, there is no implication saying that the diameter of a set $G \in \GG$ is small whenever $m(G)$ is small. This fact causes a certain discrepancy between the two issues we want to relate. Namely, $\MB$ may behave badly because of some non-local effects which do not play a role in the problem of differentiation. Thus, we formulate an additional condition on $\BB$ which makes such a situation impossible:

\begin{enumerate}[leftmargin=\parindent,align=left,labelwidth=\parindent,label=(M)]
\item there exists a set $F \subset \TT^\omega$ of full measure such that for each $g \in F$ we have 
$$
\forall_{\epsilon > 0} \ \exists_{\delta > 0} \ \forall_{E \in \BB(g)} \ \big(m(E) < \delta \big) \implies \big( {\rm diam}(E) < \epsilon \big). $$ 
\end{enumerate}

Finally, we are ready to prove the following result which summarizes the considerations in this part of the article.

\begin{theorem}\label{thm:2}
Let $\BB$ be an arbitrary differentiation basis (not necessarily of non-centered type). If $\lim_{k \rightarrow \infty} \delta_k^\BB = 0$, then $\BB$ differentiates $L^1(\TT^\omega)$. On the other hand, if $\lim_{k \rightarrow \infty} \delta_k^\BB = \delta_0 > 0$ and, additionally, $\BB$ satisfies {\rm (M)}, then $\BB$ does not differentiate $L^1(\TT^\omega)$.
\end{theorem}

\begin{proof}
First we consider the case $\lim_{k \rightarrow \infty} \delta_k^\BB = 0$. Let $f \in L^1(\TT^\omega)$ and fix $\epsilon > 0$. Our goal is to estimate from above the size of the set
$$
L_\epsilon(f) = \Big\{ g \in \TT^\omega : \exists_{B_n \Rightarrow g} \ \big( \limsup_{n \rightarrow \infty} |f_{B_n} - f(g) | > \epsilon \big) \Big\}.
$$
Take $f_\epsilon$ continuous and satisfying $\| f - f_\epsilon \| < \epsilon^2 / (10 \cdot 2^{k_\epsilon})$, where $k_\epsilon$ is such that $\delta_{k_\epsilon}^\BB \leq \frac{\epsilon}{2}$ (notice that continuous functions are dense in $L^1(\TT^\omega)$ by \cite[Proposition~7.9]{Fo}). For each $g \in \TT^\omega$ and $B \in \BB(g)$ we have the estimate
\begin{equation}\label{eq:12}
| f_B - f(g) | \leq | f_B - (f_\epsilon)_B | + | (f_\epsilon)_B - f_\epsilon(g) | + |f_\epsilon(g) - f(g)|.
\end{equation}
Thus, if $| f_B - f(g) | > \epsilon$, then at least one of the three quantities on the right hand side of \eqref{eq:12} is greater than $\frac{\epsilon}{3}$. Observe that, by continuity of $f_\epsilon$, we have $| (f_\epsilon)_B - f_\epsilon(g) | \leq \frac{\epsilon}{3}$ if ${\rm diam}(B)$ is sufficiently small. Moreover, observe that $| f_B - (f_\epsilon)_B | = | (f-f_\epsilon)_B | \leq \MB(f - f_\epsilon)(g)$. It will be convenient to introduce the auxiliary sets
$$
L_{\epsilon,1} (f) = \{ g \in \TT^\omega : | f_\epsilon(g) - f(g) | > \epsilon / 3 \}
$$
and 
$$
L_{\epsilon,2} (f) = \{ g \in \TT^\omega : \MB(f - f_\epsilon)(g) > \epsilon / 3 \}.
$$
By \eqref{eq:12} and the arguments mentioned above we see that $ L_\epsilon(f) \subset L_{\epsilon,1} (f) \cup L_{\epsilon,2} (f)$. We now estimate the sizes of $L_{\epsilon,1} (f)$ and $L_{\epsilon,2} (f)$, respectively. First, it is easy to see that
$$
m(L_{\epsilon,1} (f))  \leq \frac{3}{\epsilon} \, \|f - f_\epsilon \|_1 \leq \frac{\epsilon}{2}.
$$ 
Moreover, we have $ m(L_{\epsilon,2} (f)) \leq \epsilon / 2$. Indeed, if $ m(L_{\epsilon,2} (f)) > \frac{\epsilon}{2}$, then
$$
\frac{\epsilon}{3} \, m(L_{\epsilon,2} (f)) \geq \frac{\epsilon^2}{6} > 2^{k_\epsilon} \| f - f_\epsilon \|_1,
$$ 
which contradicts the assumption $\delta_{k_\epsilon}^\BB \leq \frac{\epsilon}{2}$. Consequently, we obtain $m(L_\epsilon(f)) < \epsilon$ and, since $\epsilon$ was arbitrary, we conclude that for a.e. $g \in \TT^\omega$,
$$
\lim_{n \rightarrow \infty} f_{B_n} = f(g), \qquad B_n \Rightarrow g. 
$$ 
Therefore, $\BB$ differentiates $L^1(\TT^\omega)$. 

Next, consider the case $\lim_{k \rightarrow \infty} \delta_k^\BB = \delta_0 > 0$ and assume that (M) is satisfied. For each $n \in \NN$ we can find $\lambda_n > 0$, $E_n \subset \TT^\omega$ and a non-negative function $f_n \in L^1(\TT^\omega)$ such that
$$
E_n = \{g \in \TT^\omega : \MB f_n(g) > \lambda_n \}
$$
and
$$
\lambda_n  m(E_n) > \frac{2^{n+2}}{\delta_0} \, \| f_n \|_1 \quad {\rm and } \quad m(E_n) > \frac{\delta_0}{2}. 
$$
In addition, we can assume that $\| f_n \|_1 = \delta_0 / 2^{n+2}$. Indeed, if $\| f_n \|_1 \neq \delta_0 / 2^{n+2}$, then we simply replace $f_n$ by $\alpha f_n$ with $\alpha = \delta_0 / (2^{n+2} \|f_n\|_1)$. Observe that $\lambda_n > 1$ since $m(E_n) \leq 1$. Finally, we denote $f = \sum_{n \in \NN} f_n$ and $E = \bigcap_{n \in \NN} \bigcup_{i \geq n} E_i$. For each $g \in E$ there exists $\{ B_n : n \in \NN\} \subset \BB(g)$ such that $\lim_{n \rightarrow \infty} m(B_n) = 0$ and
$$
\limsup_{n \rightarrow \infty} f_{B_n} \geq \limsup_{n \rightarrow \infty} (f_n)_{B_n} \geq 1. 
$$ 
Moreover, by (M) we get $B_n \Rightarrow g$ provided that $g \in E \cap F$, where $F$ is the set specified in the statement of condition (M). Thus, for the set 
$$
E_f = \big\{ g \in \TT^\omega : \exists_{B_n \Rightarrow g} \ \big(\limsup_{n \rightarrow \infty} f_{B_n} \geq 1 \big)   \big\}
$$
we have the estimate
$$
m(E_f) \geq m(E \cap F) = m(E) \geq \liminf_{n \rightarrow \infty} m(E_n) \geq \frac{\delta_0}{2}, 
$$ 
where in the second inequality we used the fact that $m$ is finite. On the other hand, since $\| f \|_1 = \sum_{n \in \NN} \|f_n \|_1 = \delta_0 / 4$, the set $A = \{ g \in \TT^\omega : f(g) \geq 1\}$ has measure at most $\delta_0 / 4$. Therefore,
$$
m(E_f \setminus A) \geq \frac{\delta_0}{4} > 0
$$
and for each $g \in E_f \setminus A$ there exists a family $\{B_n : n \in \NN \}$ such that $B_n \Rightarrow g$ and
$$
\neg \, \big( \lim_{n \rightarrow \infty} f_{B_n} = f(g) \big).
$$
Consequently, $\BB$ does not differentiate $L^1(\TT^\omega)$. 
\end{proof}

One more comment is in order here. The role of condition {\rm (M)} is to control the impact of the non-local part of $\MB$. Another approach is to use the local maximal function in the definition of $\delta_k^\BB$ (see \cite[Theorem~1.1, Chapter III]{G}, for example). Namely, consider an arbitrary sequence $(\alpha_k)_{k=1}^\infty \subset (0, \infty)$ satisfying $\lim_{k \rightarrow \infty} \alpha_k = 0$. We let $\tilde{\delta}_k^\BB = \sup_{E \in \tilde{\EEE}_k} m(E)$, where $\tilde{\EEE}_k$ is defined as $\EEE_k$ with $\MB$ replaced by $\MB_{\alpha_k}$. The following version of Theorem \ref{thm:2} relates the problem of differentiation to the behavior of $(\tilde{\delta}_k^\BB)_{k=1}^\infty$.

\begin{theorem4.3}\label{thm:2'}
Let $\BB$ be an arbitrary differentiation basis (not necessarily of non-centered type). Then $\BB$ differentiates $L^1(\TT^\omega)$ if and only if $\lim_{k \rightarrow \infty} \tilde{\delta}_k^\BB = 0$.	
\end{theorem4.3}

\begin{proof}
The proof is very similar to the proof of Theorem \ref{thm:2}. We only sketch the needed changes.

Assume $\lim_{k \rightarrow \infty} \tilde{\delta}_k^\BB = 0$ and fix $\epsilon > 0$ and $f \in L^1(\TT^\omega)$. We define $L_\epsilon(f)$, $f_\epsilon$, and $L_{\epsilon,1}(f)$ as before with $k_\epsilon$ such that $\tilde{\delta}_{k_\epsilon}^\BB \leq \frac{\epsilon}{2}$. We also let
$$
L_{\epsilon,2}^{(k)} (f) = \{ g \in \TT^\omega : \MB_{\alpha_k}(f - f_\epsilon)(g) > \epsilon / 3 \}, \qquad k \in \NN.
$$ 	
Since $\lim_{k \rightarrow \infty} \alpha_k = 0$, we obtain
$$
L_\epsilon(f) \subset L_{\epsilon,1}(f) \cup \bigcap_{k \in \NN} \bigcup_{i \geq k} L_{\epsilon,2}^{(i)} (f) = L_{\epsilon,1}(f) \cup \bigcap_{k \in \NN} L_{\epsilon,2}^{(k)} (f) \subset L_{\epsilon,1}(f) \cup L_{\epsilon,2}^{(k_\epsilon)} (f).
$$
Then it suffices to see that $m(L_{\epsilon,2}^{(k_\epsilon)} (f)) \leq \frac{\epsilon}{2}$ and, consequently, $m(L_{\epsilon}(f)) \leq \epsilon$.

Now assume $\lim_{k \rightarrow \infty} \tilde{\delta}_k^\BB = \tilde{\delta}_0 > 0$. We construct $f$, $E_f$, and $A$ as before, using $\tilde{\delta}_0$ instead of $\delta_0$. The only modification is that now for each $n \in \NN$ the set
$$
\tilde{E}_n = \{g \in \TT^\omega : \MB_{\alpha_k} f_n(g) > \lambda_n \}
$$
plays the role of $E_n$ (in particular, $m(\tilde{E}_n) \geq \frac{\tilde{\delta}_0}{2}$). Finally, we set $\tilde{E} = \bigcap_{n \in \NN} \bigcup_{i \geq n} \tilde{E}_i$ and observe that for each $g \in \tilde{E}$ there exists $B_n \Rightarrow g$ such that
$$
\limsup_{n \rightarrow \infty} f_{B_n} \geq \limsup_{n \rightarrow \infty} (f_n)_{B_n} \geq 1. 
$$
Consequently, $m(E_f \setminus A) \geq \frac{\tilde{\delta}_0}{4}> 0$.
\end{proof}

We end our discussion with the following remark. In the proof of Theorem \ref{thm:2} we referred to the measure space only twice, namely, when we used the fact that the set of continuous functions is dense in $L^1$ and when we needed our measure to be finite. Thus, in fact, the conclusion of Theorem~\ref{thm:2} (and hence also Theorem~4.3') remains true if one replaces $\TT^\omega$ with any space for which the two conditions mentioned above are satisfied. 

\subsection*{Acknowledgements}

The author is grateful to Luz Roncal for a fruitful
discussion which inspired him to write this article. The author would also like to thank the anonymous referees for their careful reading of the manuscript and their valuable suggestions which led to an improvement of the presentation.

\end{document}